\documentclass[11pt,a4paper,reqno]{amsart}
\usepackage{cite}

\usepackage{amsmath,amssymb,amsfonts,amsthm,epsfig,graphicx,color}
\usepackage{url}

\def\a{\alpha}
\def\s{\sigma}
\newtheorem{theorem}{Theorem}[section]
\newtheorem{lemma}[theorem]{Lemma}

\newtheorem{remark}[theorem]{Remark}
\newtheorem*{theorem*}{Theorem}
\newtheorem*{corollary*}{Corollary}

\theoremstyle{definition}

\numberwithin{equation}{section}
\numberwithin{figure}{section}

\renewcommand{\le}{\leqslant}

\renewcommand{\ge}{\geqslant}

\def\R{\mathbb R}
\def\H{\mathcal H}
\def\Om{\Omega}
\def\pa{\partial}
\def\ov{\overline}

\title[Contact angle in a nonlocal capillarity problem]{Asymptotic expansions of the contact angle \\ in
nonlocal capillarity problems}

\author{Serena Dipierro}
\address[Serena Dipierro]{School of Mathematics and Statistics,
University of Melbourne,
813 Swanston Street, Parkville VIC 3010, Australia}
\email{s.dipierro@unimelb.edu.au}

\author{Francesco Maggi}
\address[Francesco Maggi]{Abdus
Salam International Center for Theoretical Physics, Strada Costiera
11, I-34151, Trieste, Italy. On leave from the
University of Texas at Austin}
\email{fmaggi@ictp.it}

\author{Enrico Valdinoci}
\address[Enrico Valdinoci]{School of Mathematics and Statistics,
University of Melbourne,
813 Swanston Street, Parkville VIC 3010, Australia, and
Weierstra{\ss} Institut f\"ur Angewandte Analysis und Stochastik,
Mohrenstra{\ss}e 39, 10117 Berlin, Germany,
and Dipartimento di Matematica, Universit\`a degli studi di Milano,
Via Saldini 50, 20133 Milan, Italy}
\email{enrico@mat.uniroma3.it}

\begin{document}

\begin{abstract}
{\rm We consider a family of nonlocal capillarity models, where surface tension
is modeled by exploiting the family of fractional interaction kernels $|z|^{-n-s}$,
with $s\in(0,1)$ and $n$ the dimension of the ambient space. The
fractional Young's law (contact angle condition) predicted by these models coincides, in the limit as $s\to 1^-$,
with the classical Young's law determined by the Gauss free energy. Here we refine this asymptotics by showing that,
for $s$ close to $1$, the fractional contact angle is always {\it smaller} than its classical counterpart when the relative adhesion
coefficient $\s$ is negative, and {\it larger} if $\s$ is positive. In addition, we address the asymptotics of the fractional Young's law
in the limit case $s\to 0^+$ of interaction kernels with heavy tails. Interestingly, near $s=0$, the dependence of
the contact angle from the relative adhesion coefficient becomes linear.}
\end{abstract}

\keywords{Nonlocal surface tension, contact angle, asymptotics.}
\subjclass[2010]{76B45, 76D45, 45M05.}

\maketitle

\section{Introduction}

In this paper we consider the family of nonlocal capillarity problems recently introduced in \cite{MaggiValdinoci},
and provide a detailed description of the asymptotic behavior of the fractional Young's law
in the two limit cases defined by this family of problems.

Let us recall that the basic model for capillarity phenomena is based on the study of the Gauss free energy \cite{Finn}
\begin{equation}
  \label{gauss}\H^{n-1}(\Om\cap\pa E)+\s\,\H^{n-1}(\pa\Om\cap\pa E)+g\,\rho\,\int_E\,x_n\,dx
\end{equation}
associated to the region $E$ occupied by a liquid droplet confined in a container $\Omega\subset\R^n$, $n\ge 2$.
In this way, $\H^{n-1}(\Om\cap\pa E)$ is the surface tension energy of the liquid interface interior to
the container, $\s\,\H^{n-1}(\pa\Om\cap\pa E)$ is the surface tension energy of the liquid interface at the boundary walls
of the container, and $g\,\rho\,\int_E\,x_n\,dx$ is the potential energy due to gravity. The mismatch between the
surface tensions of the liquid/air and liquid/solid interfaces is taken into account by the relative adhesion coefficient
$\s\in(-1,1)$. There are also situations where one wishes to consider more general potential energies, and thus the potential
energy density $g\,\rho\,x_n$ is replaced by some generic density $g(x)$.

The family of nonlocal capillarity models introduced in \cite{MaggiValdinoci} replaces the use of surface area
to define the total surface tension energy of the droplet $E$, with the fractional interaction energy
\[
I_s(E,E^c\cap\Om)+\s\,I_s(E,\Om^c)\,.
\]
Here $E^c$ stands for $\R^n\setminus E$ and given two disjoint subsets $A$ and $B$ of $\R^n$ we set
\[
I_s(A,B)=\int_A\,dx\int_B\frac{dy}{|x-y|^{n+s}}\,,\qquad s\in(0,1)\,.
\]
This kind of fractional interaction kernel has been used since a long time. Of particular interest for us is the
result of \cite{Bourgain01anotherlook,davilaBV} (see also \cite{caffarellivaldinoci,
ambrosiodephilippismartinazzi}) showing that, as $s\to1^-$ and for a suitable dimensional constant
$c(n)$, $(1-s)\,I_s(E,E^c)\to c(n)\,\H^{n-1}(\pa E)$ whenever $E$ is an open set with
Lipschitz boundary (and more generally, for every set of finite perimeter if $\H^{n-1}(\pa E)$ is replaced by
the distributional perimeter of $E$). Starting from this result one can show (see \cite[Proposition 1.2]{MaggiValdinoci})
that, similarly, as $s\to 1^-$,
\begin{equation}\label{FG:A1}
I_s(E,E^c\cap\Om)+\s\,I_s(E,\Om^c)\to c(n)\,\Big(\H^{n-1}(\Om\cap\pa E)+\s\,\H^{n-1}(\pa\Om\cap\pa E)\Big)
\end{equation}
whenever $E$ is a Lipschitz subset of $\Om$.

In general, the nonlocal interaction $I_s$ plays a role of
a fractional interpolation between classical perimeter and
Lebesgue measure, and so, in a sense, it bridges ``classical surface tensions''
to ``bulk energies of volume type''.
More precisely, as $s\to0^+$, one has that $s\,I_s(E,E^c)$
converges to
the Lebesgue measure of $E$ (up to normalization constants), and the
fractional
perimeter of a set in a domain, as introduced in \cite{caffaroquesavin},
approaches a weighted convex combination between the
Lebesgue measure of the set in the domain
and the Lebesgue measure of the complement of the set in the domain,
where the convex interpolation parameter takes into account
the behavior of the set at infinity (see \cite{MR1940355, MR3007726}
and Appendix A in \cite{2016arXiv160706872D}).
In this sense, the counterpart of~\eqref{FG:A1} as $s\to0^+$,
for smooth and bounded sets~$E\subseteq\Omega$ reads
\begin{equation}\label{FG:A0}
I_s(E,E^c\cap\Om)+\s\,I_s(E,\Om^c)\to
\bar{c}(n)\,\sigma\,|E|,
\end{equation}
for a suitable $\bar{c}(n)>0$. A proof of this will be given in
Appendix~\ref{876}.\medskip

As a matter of fact, we stress that
the case of $s\to0^+$ is always somewhat
delicate, since the regularity theory may degenerate (see \cite{MR3107529,
daviladelpinowei}), the oscillations of the set at infinity may prevent the
existence of limit behaviors (see
Examples 1 and 2 in \cite{MR3007726}), the nonlocal mean curvature
of bounded domains converges to an absolute constant independent of
the geometry involved (see Appendix B in \cite{2016arXiv160706872D})
and minimal sets completely stick to the boundary (see \cite{2015arXiv150604282D}).
\medskip

{F}rom the point of view of applications, nonlocal interactions and
fractional perimeters have also very good potentialities
in the theory of image reconstruction, since the
the numerical errors produced by the approximation of
nonlocal interactions are typically considerably smaller than the ones
related to the classical perimeter (see e.g. the discussion next to
Figures 1 and 2 in \cite{2016arXiv160706872D}).
\medskip

With these motivations in mind, in \cite{MaggiValdinoci} we considered the study
of the family of free energies
\begin{equation}
  \label{fractional gauss}
  I_s(E,E^c\cap\Om)+\s\,I_s(E,\Om^c)+\int_E\,g(x)\,dx
\end{equation}
parameterized by $s\in(0,1)$, with particular emphasis on the limit case $s\to 1^-$. In fact, the full range of values
$s\in(0,1)$ has a clear geometric interest when $\s=0$ and $g\equiv 0$. The reason is that the volume-constrained minimization
of $I_s(E,E^c\cap\Om)$ defines a fractional relative isoperimetric problem which fits naturally in the emerging
theory of fractional geometric variational problems, initiated by the seminal paper \cite{caffaroquesavin} on fractional
perimeter minimizing boundaries.

We mention that models related to the functional in~\eqref{fractional gauss}
have been
numerically analized working with Gaussian interaction kernels, see~\cite{2016arXiv160204688X}
and references therein.
\medskip

We now come the main point discussed in this paper, which is the precise behavior
of the Euler-Lagrange equation
of the fractional Gauss free energies \eqref{fractional gauss} in the limit cases $s\to 1^-$ and $s\to 0^+$. Let us recall
the important notion of {\it fractional mean curvature} of an open set $E$ with Lipschitz boundary
\[
H^s_E(x)={\rm p.v.}\int_{\R^n}\frac{(1_{E^c}-1_E)(y)}{|y-x|^{n+s}}\,dy\,,\qquad x\in\pa E\,,
\]
which was introduced and studied from a geometric viewpoint in \cite{caffaroquesavin}.
If $g\in C^1(\R^n)$ and
$E$ is a volume-constrained critical point of the fractional Gauss free energy \eqref{fractional gauss}
such that $\Om\cap\pa E$ is of class $C^{1,\a}$ for some $\a\in(s,1)$, then it was proved in \cite[Theorem 1.3]{MaggiValdinoci}
that along $\pa E$  the following Euler-Lagrange equation
\begin{equation}
  \label{el fractional}
  H^s_E(x)+g(x)=\lambda+(1-\s)\,\int_\Om^c\frac{dy}{|x-y|^{n+s}}\qquad\forall x\in\Om\cap\pa E\,,
\end{equation}
holds, where $\lambda\in\R$ is a constant Lagrange multiplier.
Let us recall that in the classical case, the Euler-Lagrange equation for the volume-constrained critical
points of the Gauss free energy takes the form
\begin{eqnarray}\label{el classic 1}
  H_E(x)+g(x)&=&\lambda\qquad\forall x\in\Om\cap\pa E\,,
  \\\label{el classic 2}
  \nu_E(x)\cdot\nu_\Om(x)&=&\s\qquad\forall x\in\pa\Om\cap\overline{\Om\cap\pa E}\,,
\end{eqnarray}
where $H_E(x)$ is the mean curvature of $\pa E$ with respect to the outer unit normal $\nu_E$ to $E$ and, again,
$\lambda$ is a Lagrange multiplier. Equation \eqref{el classic 2} is the classical Young's law, which relates the contact
angle between the interior interface and the boundary walls of the container with the relative adhesion coefficient.

An interesting qualitative feature of the fractional model is that the two well-known equilibrium equations \eqref{el classic 1}
and \eqref{el classic 2} are now merged into the same equation \eqref{el fractional}. In the fractional equation the
effect of the relative adhesion coefficient is present not only on the boundary of the wetted region, but also at
the interior interface points, because of the term
\[
(1-\s)\,\int_\Om^c\frac{dy}{|x-y|^{n+s}}\,\qquad x\in\Om\,.
\]
Notice that this term is increasingly localized near $\pa\Om$ the closer $s$ is to $1$. Moreover,
in \cite[Theorem 1.4]{MaggiValdinoci}, we have shown that \eqref{el fractional} implicitly enforces
a contact angle condition, in the sense that, if $\ov{\Om\cap\pa E}$ is a $C^{1,\a}$-hypersurface
with boundary having all of its boundary points contained in $\pa\Om$, then
\begin{equation}
  \label{nonlocal youngs law ts}
  \nu_E(x)\cdot\nu_\Om(x)=\cos(\pi-\theta(s,\s))\qquad\forall x\in\pa\Om\cap\overline{\Om\cap\pa E}\,.
\end{equation}
Here $\theta\in C^\infty((0,1)\times(-1,1);(0,\pi))$ is implicitly defined by the equation
\begin{equation*}
1+\sigma =(\sin\theta)^s\,\frac{M(\theta,s)}{M\left(\frac\pi2,s\right)}\,,
\end{equation*}
where
\[
M(\theta,s)=2\int_0^\a\left[
\int_0^{+\infty}
\frac{r\,dr}{
{\left( {r}^{2}+2\,r\,\cos t+1\right) }^{\frac{2+s}{2}}}
\right]\,dt\,.
\]
Thus, as for the classical case,
the contact angle does not depend on potential energy density~$g$. Also, as~$s\to1^-$, the fractional contact angle
converges to the one predicted by the classical Young's Law \eqref{el classic 2}, namely, see again \cite[Theorem 1.4]{MaggiValdinoci},
\begin{equation}\label{YY:s1}
\lim_{s\to1^-} \theta(s,\sigma)=\arccos(-\sigma)\,.
\end{equation}
The goal of this paper is to provide precise asymptotics
for $\theta(s,\s)$ both as~$s\to1^-$ and as~$s\to0^+$.
As $s\to0^+$, the relation between $\theta$ and $\sigma$ changes dramatically, and the trigonometric identity \eqref{YY:s1}
is replaced by the {\it linear} relation
\begin{equation}\label{YY:s0}
\lim_{s\to0^+} \theta(s,\sigma)=\frac{\pi}{2}(1+\sigma)\,.
\end{equation}
Formulas~\eqref{YY:s1} and~\eqref{YY:s0}
are indeed part of a more general result, which goes as follows:

\begin{theorem}\label{ASY}
If~$\theta(s,\sigma)$ is the angle prescribed by the fractional Young's law~\eqref{nonlocal youngs law ts},
then $\theta(s,\cdot)$ is strictly increasing on $(-1,1)$ for every $s\in(0,1)$, and for every $\s\in(-1,1)$ we have
\begin{equation}\label{LAPRIMA}
\begin{split}
\theta(s,\sigma)\,=\,&
\arccos(-\sigma)\\ &-
\frac{2\sigma\log 2+(1-\sigma)\log(1-\sigma)-
(1+\sigma)\log(1+\sigma)
}{2\sqrt{1-\sigma^2}}\,(1-s)\\ &+o(1-s)
\,.\end{split}
\end{equation}
in the limit $s\to 1^-$, and
\begin{equation}\label{LASECONDA}\begin{split}
\theta(s,\sigma)\,=\,&\frac{\pi}{2}(1+\sigma)
\\ &-\left[\frac\pi2\,(1+\sigma)\,\log\left(
\cos\frac{\pi\sigma}{2}\right)-
\Xi\left(\frac{\pi}{2}\,(1+\sigma)\right)
+(1+\sigma)\,\Xi\left(\frac{\pi}{2}\right)\right]\, s\\ &+o(s)
\,,\end{split}
\end{equation}
in the limit $s\to 0^+$. Here we set
\begin{equation}\label{XIDE} \Xi(\a):=
\int_0^\a
\frac{t}{\tan t}
\,dt\,,\qquad\a\in[0,\pi]\,.
\end{equation}
\end{theorem}

\begin{remark}
  {\rm Though not crucial for our computations, we observe that
$$ \Xi\left(\frac{\pi}{2}\right)=\frac\pi2 \,\log2.$$}
\end{remark}

\begin{remark}
  {\rm We recall from \cite{MaggiValdinoci} that $\theta(s,0)=\pi/2$ for every $s\in(0,1)$. Thus,
  in the case
  $\s=0$ corresponding to the relative fractional isoperimetric problem,
  the fractional contact angles are all equal to the classical ninety degrees contact angle. We also
  notice that, despite this fact, when $\s=0$ and $n=2$ {\it half-disks are never volume-constraints critical points of
  the fractional capillarity energy on a half-space}.
  A geometric proof of this fact will be given in Appendix~\ref{09:SEC}.}
\end{remark}

\begin{remark}
  {\rm It is easily seen that the function
  \[
  2\sigma\log 2+(1-\sigma)\log(1-\sigma)-(1+\sigma)\log(1+\sigma)
  \]
  is strictly concave on $(-1,0)$, strictly convex on $(0,1)$, and that it takes the value $0$ at $\s=0,1,-1$.
  As a consequence, equation \eqref{LAPRIMA} implies that
  \[
  \theta(s,\s)<\arccos(-\s)\qquad\forall\s\in(-1,0)\,,\qquad
  \theta(s,\s)>\arccos(-\s)\qquad\forall\s\in(0,1)\,,
  \]
  provided $s$ is close enough to $1$. Correspondingly, for $s$ close to $1$,
  in the hydrophilic regime $\s\in(-1,0)$ fractional
  droplets are more hydrophilic than their classical counterparts, while in the hydrophobic case they are more
  hydrophobic. As
  \[
  \frac{\pi}{2}(1+\sigma)<\arccos(-\s)\qquad\forall\s\in(-1,0)\qquad
    \frac{\pi}{2}(1+\sigma)>\arccos(-\s)\qquad\forall\s\in(0,1)\qquad
  \]
  the same assertions hold for $s$ close to $0$. Figure~\ref{fig PLOT},
  \begin{figure}
  \epsfig{file=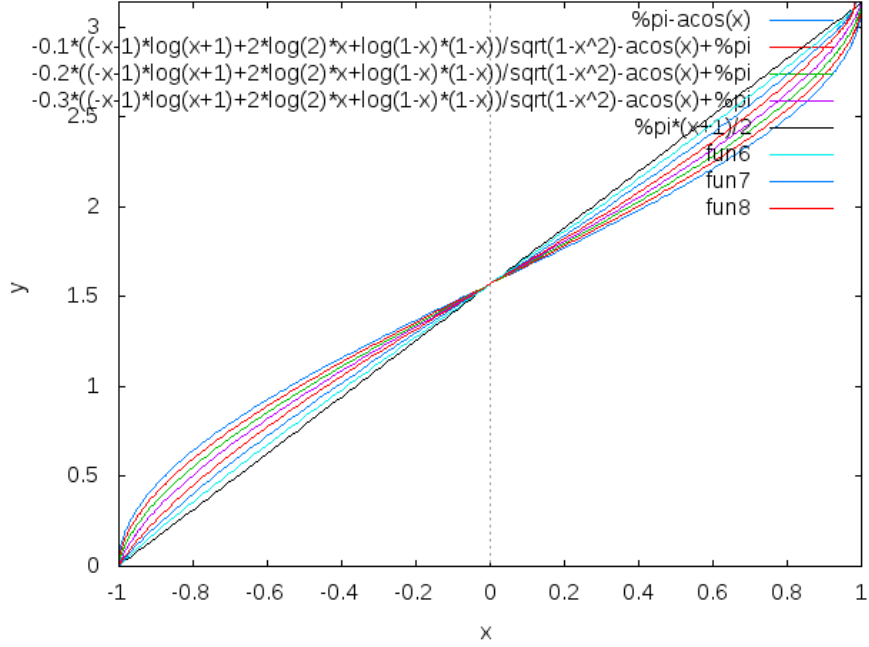,height=9cm}
 \caption{{\small An approximate plot of the function~$(-1,1)\ni\sigma\mapsto\theta(s,\sigma)$
for some values of~$s$. Notice that ``the arccosine linearizes''
as~$s$ goes from~$1$ to~$0$. The approximate picture is obtained by
plotting with Maxima formulas \eqref{LAPRIMA} (disregarding $o(1-s)$)
and \eqref{LASECONDA} (disregarding $o(s)$) for some values of $s$.
Also, to approximatively plot the function $\Xi$ in \eqref{XIDE},
a high order Taylor expansion of the function $t/\tan t$
at $t=\frac\pi2$ has been used.}}\label{fig PLOT}
\end{figure}
  which is also included with the aim
  of facilitating the interpretation of the mathematical results in~\eqref{LAPRIMA} and~\eqref{LASECONDA},
  suggests that this should be the case for every $s\in(0,1)$.}
\end{remark}

\medskip

\noindent {\bf Acknowledgments:} 
It is a pleasure to thank Matteo Cozzi for his very
interesting comments on a preliminary version of this paper.

FM was supported by NSF-DMS Grant 1265910 and NSF-DMS FRG Grant 1361122.
EV was supported by ERC grant 277749 ``E.P.S.I.L.O.N. Elliptic PDE's and Symmetry of Interfaces and Layers for Odd Nonlinearities''
 and PRIN grant 201274FYK7 ``Critical Point Theory and Perturbative Methods for Nonlinear Differential Equations''.

\section{Proof of Theorem \ref{ASY}} We start with the following lemma.

\begin{lemma}
If $\a\in\left(0,\frac\pi2\right)$, then
\begin{equation}\label{AP}
\int_0^\a\left[
\int_0^{+\infty}
\frac{\log\left( {r}^{2}+2\,r\,\cos t+1\right) }{
{\left( {r}^{2}+2\,r\,\cos t+1\right) }^{\frac{3}{2}}}\,r\,dr
\right]\,dt =
\frac{4(1-\cos\alpha)+
2\big(\log(\cos\alpha+1)-\log2\big)
}{\sin \alpha}\,
.\end{equation}
\end{lemma}

\begin{proof} To show this,
given $a$, $r>0$, we let
\begin{eqnarray*}
&& \rho(a,r):=\sqrt{ {r}^{2}+2\,a\,r+1 }\\
{\mbox{and }}&& g(a,r):=
\frac{2\left( a\,r+1\right) \,\log\rho(a,r)-
2\,a\,\rho(a,r)
\,\log
\left( a+r+\rho(a,r)
\right) +2\,a\,r+2}{
({a}^{2}-1)\,\rho(a,r)
}
.\end{eqnarray*}
By a direct computation, we see that
$$ \partial_r g(a,r)=
\frac{r\,\log\left( {r}^{2}+2\,a\,r+1\right) }{
{\left( {r}^{2}+2\,a\,r+1\right) }^{\frac{3}{2}}}.$$
Also,
$$ g(a,0)= \frac{-2a\log(a+1)+2}{a^2-1}.$$
In addition, as $r\to+\infty$,
$$ \rho(a,r)=r\left(1+\frac{a}{r}+O\left(\frac1{r^2}\right)\right),$$
and so
\begin{eqnarray*}&& g(a,r)\\ &=&
\frac{1}{
({a}^{2}-1)\,
r\,\left(1+\frac{a}{r}+O\left(\frac1{r^2}\right)\right)}
\,\left[2\left( a\,r+1\right) \,\log\left(
r\left(1+\frac{a}{r}+O\left(\frac1{r^2}\right)\right)
\right)\right.
\\ && \qquad\left.-2\,a\,
r\,\left(1+\frac{a}{r}+O\left(\frac1{r^2}\right)\right)\,\log
\left( a+r+r\left(1+\frac{a}{r}+O\left(\frac1{r^2}\right)\right)
\right) +2\,a\,r+2\right]
\\
&=&
\frac{1}{
({a}^{2}-1)\,
r\,\left(1+\frac{a}{r}+O\left(\frac1{r^2}\right)\right)}
\,\left[2\left( a\,r+1\right) \,\left[ \log r +\log\left(
1+\frac{a}{r}+O\left(\frac1{r^2}\right)\right)
\right]\right.
\\ && \qquad\left.-2\,a\,
r\,\left(1+\frac{a}{r}+O\left(\frac1{r^2}\right)\right)\,
\left[
\log r+\log
\left(2+\frac{2a}{r}+O\left(\frac1{r^2}\right)\right)
\right] +2\,a\,r+2\right]
\\
&=&
\frac{1}{
({a}^{2}-1)\,
\left(1+o(1)\right)}
\,\left[o(1)
-2\,a\,\log(2+o(1))+2\,a\right]
\\ &&\longrightarrow \frac{2a(1-\log 2)}{a^2-1}\qquad{\mbox{ as }}r\to+\infty.
\end{eqnarray*}
As a consequence,
\begin{equation}\label{a cost}
\begin{split}
&
\int_0^{+\infty}
\frac{\log\left( {r}^{2}+2\,a\,r+1\right) }{
{\left( {r}^{2}+2\,a\,r+1\right) }^{\frac{3}{2}}}\,r\,dr=
\int_0^{+\infty} \partial_r g(a,r)\,dr=g(a,+\infty)-g(a,0)\\
&\qquad\qquad=
\frac{2a\big(1-\log 2+\log(a+1)\big)-2}{a^2-1}=:h(a).
\end{split}\end{equation}
Now, for any $\a\in\left(0,\frac\pi2\right)$, we set
$$ H(\a):=\frac{4(1-\cos\alpha)+
2\big(\log(\cos\alpha+1)-\log2\big)
}{\sin \alpha}.$$
Notice that
$$ \lim_{\a\to 0^+} H(\a)=0.$$
Also, we compute that
$$ \partial_\a H(\a)= h(\cos\a).$$
Therefore
$$ \int_0^\a h(\cos t)\,dt = H(\a)-H(0)=\frac{4(1-\cos\alpha)+
2\big(\log(\cos\alpha+1)-\log2\big)
}{\sin \alpha}.$$
Using this and \eqref{a cost} with $a:=\cos t$, we obtain \eqref{AP}.
\end{proof}

\begin{proof}[Proof of Theorem~\ref{ASY}] Since~\eqref{LAPRIMA} and~\eqref{LASECONDA}
(as well as the monotonicity property in~$\sigma$
claimed in Theorem~\ref{ASY}) are invariant for the symmetry
$$ \theta(s,-\sigma)=\pi-\theta(s,\sigma),$$
we can assume that~$\sigma\in(-1,0)$,
and thus~(see \cite{MaggiValdinoci}) $\theta(s,\sigma)\in
\left(0,\frac\pi2\right)$. Also, whenever clear
from the context, we use the short notation~$\theta(s):=
\theta(s,\sigma)$.
For any
$\alpha\in\left(0,\frac\pi2\right]$ we define the cone of opening $2\alpha$
$$ \Gamma_\alpha:=\{ x=(x_1,x_2)\in\R^2 {\mbox{ s.t. }} x_2<-|x|\,\cos\alpha\}.$$
Then, for any $s\in(0,1]$ and
$\eta\in[0,+\infty)$, we set
\begin{equation}\label{IMV:0}
I(\eta,\alpha,s):=
\int_{\Gamma_\alpha} \frac{dz}{|\eta e_2-z|^{2+s}}.\end{equation}
We know that, see \cite[Proof of Theorem 1.4, step three]{MaggiValdinoci},
\begin{equation}\label{IMV}
I(1,\alpha,s)=2\,
\int_0^\a\left[
\int_0^{+\infty}
\frac{r\,dr}{
{\left( {r}^{2}+2\,r\,\cos t+1\right) }^{\frac{2+s}{2}}}
\right]\,dt\end{equation}
and that
the optimal angle $\theta=\theta(s)$ in the Young's law
satisfies
\begin{equation*}
1+\sigma = \frac{I(1,\theta,s)}{I\left(\sin\theta,\frac\pi2,s\right)}.
\end{equation*}
Also, by scaling~\eqref{IMV:0}, one sees that
\begin{equation*}
I\left(\sin\theta,\frac\pi2,s\right) =
\frac{ I\left(1,\frac\pi2,s\right) }{(\sin\theta)^s}\end{equation*}
and thus setting
$$ f(s,\a):=\frac{(\sin\a)^s\, I(1,\a,s) }{ I\left(1,\frac\pi2,s\right) }$$
we have
\begin{equation}\label{LASIG}
f(s,\theta(s))=1+\sigma.\end{equation}
Accordingly, recalling that~$\theta=\theta(s)$ and differentiating with respect to $s$,
we find that
\begin{equation*}
\partial_s f(s,\theta(s)) +\partial_\a f(s,\theta(s)) \,\partial_s\theta(s)=0,
\end{equation*}
and so (being, evidently, $\pa_\a f>0$ for $\a\in(0,\pi/2)$)
\begin{equation}\label{DINI}
\partial_s\theta(s)=-\frac{
\partial_s f(s,\theta(s)) }{\partial_\a f(s,\theta(s))}\,.
\end{equation}
Now, to compute $\partial_s\theta(1)$ and
prove~\eqref{LAPRIMA}, we evaluate~$
\partial_s f(1,\theta(1))$ and~$\partial_\a f(1,\theta(1)) $ and we substitute
these expressions into~\eqref{DINI} (evaluated at~$s=1$). For this, we recall (see~\cite[end of section~4]{MaggiValdinoci}) that
\begin{equation}\label{I111} I(1,\alpha,1)=\frac{2\sin\alpha}{1+\cos\alpha}\end{equation}
and therefore
\begin{equation}\label{I11} \partial_\a I(1,\a,1)=\frac{2}{1+\cos\a}.\end{equation}
In addition,
\begin{equation}\label{conS}
\partial_\a f(s,\a)=
\frac{ s(\sin\a)^{s-1}\cos\a\, I(1,\a,s)
+(\sin\a)^s\, \partial_\a I(1,\a,s) }{ I\left(1,\frac\pi2,s\right)}
\end{equation}
so that, by~\eqref{I11},
\begin{equation}\label{AP000}
\partial_\a f(1,\a)=
\frac{\frac{2\sin\a\,\cos\a}{1+\cos\a}+\frac{2\sin\a}{1+\cos\a}}{ 2 }
= \sin\a\,.
\end{equation}
On the other hand,
\begin{equation}\label{AP2:pre}
\partial_s f(s,\a)
=
\frac{(\sin\a)^s\,\log
(\sin\a)\, I(1,\a,s) }{ I\left(1,\frac\pi2,s\right) }+
\frac{
(\sin\a)^s\, \partial_s I(1,\a,s) }{ I\left(1,\frac\pi2,s\right) }
-
\frac{
(\sin\a)^s\, I(1,\a,s)\,
\partial_s I\left(1,\frac\pi2,s\right)
}{\left( I\left(1,\frac\pi2,s\right)\right)^2 }
,\end{equation}
and so, recalling~\eqref{I111},
\begin{equation}\label{AP2}
\partial_s f(1,\a)=
\frac{\log(\sin\a)\sin^2\a}{1+\cos\a}
+ \frac{\sin\a}2\,\partial_s I(1,\a,1)
-\frac{\sin^2\a}{2\,(1+\cos\a)}\,\partial_s I\left(1,\frac\pi2,1\right).
\end{equation}
Furthermore, by~\eqref{IMV},
\begin{equation}\label{nao}
\partial_s I(1,\alpha,s)=-
\int_0^\alpha\left[
\int_0^{+\infty}
\frac{\log \left( {r}^{2}+2\,r\,\cos t+1\right) r\,dr}{
{\left( {r}^{2}+2\,r\,\cos t+1\right) }^{\frac{2+s}{2}}}
\right]\,dt.\end{equation}
Hence, by~\eqref{AP},
$$ \partial_s I(1,\a,1)=-
\frac{4(1-\cos\alpha)+
2\big(\log(\cos\alpha+1)-\log2\big)
}{\sin \alpha}\,.$$
Now we insert this identity into~\eqref{AP2} and we conclude that
\begin{equation}\label{AP001}
\begin{split}
\partial_s f(1,\a) \,&=
\frac{\log(\sin\a)\,\sin^2\a}{\cos\a+1}-
2(1-\cos\a)-\big(\log(\cos\alpha+1)-\log2\big)+
\frac{(2-\log 2)\sin^2\a}{\cos\a+1}\\
&=(1-\cos\a)\big(\log(\sin\a)-\log2\big)-
\log(\cos\alpha+1)+\log2
\,.\end{split}
\end{equation}
Hence, we insert~\eqref{AP000}
and \eqref{AP001} into~\eqref{DINI}
and we find that
\[
\partial_s\theta(1)=\frac{
(\cos\theta(1)-1)\big(\log(\sin\theta(1))-\log2\big)+
\log(\cos\theta(1)+1)-\log2}{\sin\theta(1)}\,.
\]
Since we have~$\cos \theta(1)=-\sigma$ and so~$\sin\theta(1)=
\sqrt{1-\sigma^2}$, we finally conclude that
\begin{eqnarray*}
\partial_s\theta(1)&=&
\frac{
(\sigma+1)\big(\log 2-\log(\sqrt{1-\sigma^2})\big)+
\log(1-\sigma)-\log2}{\sqrt{1-\sigma^2}}\\
&=& \frac{2\sigma\log 2+2\log(1-\sigma)-(\sigma+1)\log(1-\sigma^2)
}{2\sqrt{1-\sigma^2}}
\,.
\end{eqnarray*}
Accordingly, as $s\to1^-$,
\begin{eqnarray*}
\theta(s)&=&\theta(1)+\partial_s\theta(1)\,(s-1)+o(1-s)\\
%&=& \arccos(-\sigma)-
%\frac{2\sigma\log 2+2\log(1-\sigma)-(\sigma+1)\log(1-\sigma^2)
%}{2\sqrt{1-\sigma^2}}
%\,(1-s)+o(1-s)\\
&=& \arccos(-\sigma)-
\frac{2\sigma\log 2+(1-\sigma)\log(1-\sigma)-
(1+\sigma)\log(1+\sigma)
}{2\sqrt{1-\sigma^2}}
\,(1-s)+o(1-s),
\end{eqnarray*}
which establishes~\eqref{LAPRIMA}.
\medskip

Now we prove~\eqref{LASECONDA}.
To this aim, we observe that
\begin{equation}\label{exp r1}
\left( {r}^{2}+2\,r\,\cos t+1\right)^{-\frac{2+s}{2}}
=r^{-2-s} +\frac{\kappa}{r^{3+s}},\end{equation}
where
$$\kappa=\kappa(r,t,s):=r^{3+s}
\left( {r}^{2}+2\,r\,\cos t+1\right)^{-\frac{2+s}{2}} -r=
r
\left( 1+\frac{2\,\cos t}{r}+\frac{1}{r^2}\right)^{-\frac{2+s}{2}} -r.$$
As a consequence,
$$ \int_1^{+\infty}
\frac{r\,dr}{
{\left( {r}^{2}+2\,r\,\cos t+1\right) }^{\frac{2+s}{2}}}
= \int_1^{+\infty} \left[r^{-1-s} +\frac{\kappa}{r^{2+s}}\right]\,dr
= \frac{1}{s}+\kappa_0,$$
where
\begin{equation}\label{kappa 0}
\kappa_0=\kappa_0(t,s):=
\int_1^{+\infty} \frac{\kappa}{r^{2+s}}\,dr
=
\int_1^{+\infty} \frac{\left( 1+\frac{2\,\cos t}{r}+
\frac{1}{r^2}\right)^{-\frac{2+s}{2}} -1}{
r^{1+s}}\,dr
.\end{equation}
By a first order expansion, we notice
that~$\kappa_0$
is a function
which is bounded uniformly
in~$t\in\left[0,\frac\pi2\right]$ and~$s\in(0,1)$.

Therefore
\begin{equation}\label{ka1} \int_0^{+\infty}
\frac{r\,dr}{
{\left( {r}^{2}+2\,r\,\cos t+1\right) }^{\frac{2+s}{2}}}
= \frac{1}{s}+\kappa_1,\end{equation}
where
\begin{equation}\label{kappa 1}
\kappa_1=\kappa_1(t,s):=
\int_0^1
\frac{r\,dr}{
{\left( {r}^{2}+2\,r\,\cos t+1\right) }^{\frac{2+s}{2}}}
+ \kappa_0(t,s).\end{equation}
By construction, we have
that~$\kappa_1$
is bounded uniformly
in~$t\in\left[0,\frac\pi2\right]$ and~$s\in(0,1)$.

Hence, from~\eqref{IMV},
\begin{equation} \label{RAP}
I(1,\alpha,s)=2\,
\int_0^\a\left[ \frac{1}{s}+\kappa_1
\right]\,dt = \frac{2\a}{s}+\kappa_2,\end{equation}
where
\begin{equation}\label{kappa 2}
\kappa_2=\kappa_2(\a,s):=
2\,\int_0^\a\kappa_1\,dt.\end{equation}
We remark that~$\kappa_2$
is bounded uniformly
in~$\a\in\left[0,\frac\pi2\right]$ and~$s\in(0,1)$.

Now, we observe that
\begin{equation}\label{78qAl}
\liminf_{s\to0^+}\theta(s)>0.
\end{equation}
Indeed, if, by contradiction, it holds that
$$ \lim_{k\to+\infty} \theta(s_k)=0,$$
for some infinitesimal sequence~$s_k$, then
we deduce from~\eqref{LASIG} that
$$ 1+\sigma= f(s_k,\theta(s_k))=
\lim_{k\to+\infty}
\frac{ \big(\sin(\theta(s_k)\big)^{s_k}\,I(1,\theta(s_k),s_k) }{
I\left(1,\frac\pi2,s_k\right) }
\le
\lim_{k\to+\infty}
\frac{ I(1,\theta(s_k),s_k) }{
I\left(1,\frac\pi2,s_k\right) }.$$
Therefore, by~\eqref{RAP},
$$ 0<1+\sigma\le \lim_{k\to+\infty}
\frac{ \frac{2\theta(s_k)}{s_k}+\kappa_2(\theta(s_k),s_k) }{
\frac{\pi}{s_k}+\kappa_2\left(\frac\pi2,s_k\right) } =
\lim_{k\to+\infty}
\frac{ 2\theta(s_k)+s_k\kappa_2(\theta(s_k),s_k) }{
\pi+s_k\kappa_2\left(\frac\pi2,s_k\right) } = 0,
$$
which is a contradiction, thus proving~\eqref{78qAl}.

{F}rom~\eqref{78qAl}, we deduce that
\begin{equation}\label{SINBO}
\lim_{s\to0^+} \big(\sin\theta(s)\big)^s=1.\end{equation}
Therefore, in light of~\eqref{LASIG} and~\eqref{RAP},
\begin{eqnarray*}
&& 1+\sigma= \lim_{s\to0^+} f(s,\theta(s))=
\lim_{s\to0^+}
\frac{ \big(\sin(\theta(s)\big)^{s}\,I(1,\theta(s),s) }{
I\left(1,\frac\pi2,s\right) }\\
&&\qquad=
\lim_{s\to0^+}
\frac{ I(1,\theta(s),s) }{
I\left(1,\frac\pi2,s\right) }
= \lim_{s\to0^+}
\frac{ \frac{2\theta(s)}{s}+\kappa_2(\theta(s),s) }{
\frac{\pi}{s}+\kappa_2\left(\frac\pi2,s\right) }
\\ &&\qquad=
\lim_{s\to0^+}
\frac{ 2\theta(s)+s\kappa_2(\theta(s),s) }{
\pi+s\kappa_2\left(\frac\pi2,s\right) }
=\frac{2\lim_{s\to0^+}\theta(s)}{\pi},
\end{eqnarray*}
which proves that
\begin{equation}\label{LINEAR}
\theta(0)=\frac{\pi}{2}(1+\sigma)
.\end{equation}
Now, by \eqref{kappa 0},
\begin{equation} \label{0iosak11}
\kappa_0(t,0)=
\int_1^{+\infty} \frac{\left( 1+\frac{2\,\cos t}{r}+
\frac{1}{r^2}\right)^{-1} -1}{
r}\,dr=-\int_1^{+\infty}
\frac{2\cos t+\frac1r}{r^2+2\,r\cos t+1}\,dr.\end{equation}
We also set
\begin{eqnarray*}
&&\varphi(r,t):=\frac{\arctan \left(\frac{\cos t+r}{\sin t}\right)}{\tan t}-\frac12
\log\frac{r^2 +2\,r \cos t +1}{ r^2}\\
{\mbox{and }}&& \psi(r,t):=\log r-\varphi(r,t)=
\frac12 \log(r^2+2\,r \cos t+1)-\frac{\arctan\left(\frac{\cos t+r
}{\sin t }\right)}{\tan t}
\end{eqnarray*}
and we compute that
$$ \frac{\partial}{\partial r}
\varphi(r,t)= \frac{2\cos t+\frac1r}{r^2+2\,r\cos t+1}.$$
{F}rom this and~\eqref{0iosak11}, we conclude that
\begin{equation}\label{ps61} \kappa_0(t,0)
=\varphi(1,t)-\varphi(+\infty,t)
%%%= -\frac{t}{2\tan t}-\frac12\left( \log(\cos t+1)+\log2\right)
\,.\end{equation}
We also remark that
$$ \frac{\partial}{\partial r}\psi(r,t)=
\frac{\partial}{\partial r} \big[
\log r-\varphi(r,t) \big]
=
\frac{r}{
r^{2}+2\,r\,\cos t+1 }$$
and thus
$$ \int_0^1
\frac{r\,dr}{
{ {r}^{2}+2\,r\,\cos t+1 }} =
\psi(1,t)-\psi(0,t)=-\varphi(1,t)+
\frac{\arctan
\left(\frac{\cos t}{\sin t }\right)
}{\tan t}.$$
Hence, from \eqref{kappa 1} and~\eqref{ps61},
\begin{equation}\label{suvgf19}
\begin{split}
& \kappa_1(t,0)=
\int_0^1
\frac{r\,dr}{
{ {r}^{2}+2\,r\,\cos t+1 }}
+ \kappa_0(t,0)\\&\qquad\qquad=
\frac{\arctan
\left(\frac{\cos t}{\sin t }\right)
}{\tan t}
-\varphi(+\infty,t)=
\frac{\arctan
\left(\frac{1}{\tan t }\right)
}{\tan t}
-
\frac{\pi
}{2\tan t}
%%=-\frac{t}{\tan t}
.\end{split}\end{equation}
Now we point out that, for any~$x\in\R$,
$$ \frac12-\frac1\pi \arctan
\left(\frac{1}{\tan (\pi x) }\right)=\{x\},$$
where~$\{\cdot\}$ denotes here the fractional part.
As a consequence, for any~$t\in\left[0,\frac\pi2\right]$
(or, more generally, for any~$t\in[0,\pi]$) we have
$$ \arctan
\left(\frac{1}{\tan t }\right)
-
\frac{\pi
}{2} =-\pi \left\{ \frac{t}{\pi}\right\}=-t.$$
This and~\eqref{suvgf19} say that
$$ \kappa_1(t,0)=-\frac{t}{\tan t}.$$
Therefore, in the light of~\eqref{kappa 2}
and recalling the definition in~\eqref{XIDE}, we conclude that
\begin{equation}\label{CALCO k2}
\kappa_2(\a,0)=
-2\,\int_0^\a
\frac{t}{\tan t}
\,dt= -2\Xi(\a).\end{equation}
Now we remark that
\begin{eqnarray*}&& \log\left( {r}^{2}+2\,r\,\cos t+1\right)=
\log\left( {r}^{2}\left(1+\frac{2\,\cos t}r+\frac1{r^2}\right)\right)
\\ &&\qquad\qquad=\log r^2+ \log\left(1+\frac{2\,\cos t}r+
\frac1{r^2}\right)=2\log r+\frac{\chi_0}{r},\end{eqnarray*}
where~$\chi_0=\chi_0(r,t)$ is bounded uniformly in~$r\in[1,+\infty)$
and~$t\in\left[0,\frac\pi2\right]$.
Then, by~\eqref{exp r1}, we obtain
\begin{equation}\label{sp0193}
\frac{\log\left( {r}^{2}+2\,r\,\cos t+1\right) }{
{\left( {r}^{2}+2\,r\,\cos t+1\right) }^{\frac{2+s}{2}}}\,r=
\frac{ \left(2\,r\log r+\chi_0\right)\left(1+\frac{\kappa}{r}\right)
}{ r^{2+s}  }=
\frac{ 2\,r\log r+ 2\kappa\log r+\frac{\chi_0\kappa}{r}
+\chi_0
}{ r^{2+s}  }.
\end{equation}
Furthermore,
$$ \frac{\partial}{\partial r} \frac{1+s\,\log r}{s^2\,r^s}=-\frac{\log r}{r^{1+s}}.$$
This and~\eqref{sp0193} imply that
\begin{equation*}
\int_1^{+\infty}\frac{\log\left( {r}^{2}+2\,r\,\cos t+1\right) }{
{\left( {r}^{2}+2\,r\,\cos t+1\right) }^{\frac{2+s}{2}}}\,r\,dr=
2\int_1^{+\infty} \frac{\log r}{r^{1+s}}\,dr +\chi_1
= \frac{2}{s^2}+\chi_1,
\end{equation*}
with~$\chi_1=\chi_1(t,s)$, which is bounded uniformly in~$t\in\left[0,\frac\pi2\right]$
and~$s\in(0,1)$.

Consequently,
\begin{equation*}
\int_0^{+\infty}\frac{\log\left( {r}^{2}+2\,r\,\cos t+1\right) }{
{\left( {r}^{2}+2\,r\,\cos t+1\right) }^{\frac{2+s}{2}}}\,r\,dr
= \frac{2}{s^2}+\chi_2,
\end{equation*}
with~$\chi_2=\chi_2(t,s)$, which is bounded uniformly in~$t\in\left[0,\frac\pi2\right]$
and~$s\in(0,1)$. {F}rom this identity and~\eqref{nao}, we obtain
$$ \partial_s I(1,\alpha,s)=-
\int_0^\alpha\left[ \frac{2}{s^2}+\chi_2
\right]\,dt = -\frac{2\alpha}{s^2}+\chi_3,$$
with~$\chi_3=\chi_3(\a,s)$ bounded uniformly in~$\a\in\left[0,\frac\pi2\right]$
and~$s\in(0,1)$.

Using this and \eqref{RAP}, we have
\begin{equation}\label{AP2:prebis}
\begin{split}
& \frac{
(\sin\a)^s\, \partial_s I(1,\a,s) }{ I\left(1,\frac\pi2,s\right) }
-
\frac{
(\sin\a)^s\, I(1,\a,s)\,
\partial_s I\left(1,\frac\pi2,s\right)
}{\left( I\left(1,\frac\pi2,s\right)\right)^2 } \\
=\;&
\frac{
(\sin\a)^s\, \left( -\frac{2\alpha}{s^2}+\chi_3\right) }{ \frac{\pi}{s}+\bar\kappa_2}
-
\frac{
(\sin\a)^s\, \left(\frac{2\a}{s}+\kappa_2\right)\,
\left( -\frac{\pi}{s^2}+\chi_3\right)
}{\left(\frac{\pi}{s}+\bar\kappa_2\right)^2 } \\
=\;& \frac{ (\sin\a)^s }{\left(\frac{\pi}{s}+\bar\kappa_2\right)^2 }
\left[ \left( -\frac{2\alpha}{s^2}+\chi_3\right) \left(\frac{\pi}{s}+\bar\kappa_2\right)-
\left(\frac{2\a}{s}+\kappa_2\right)\,
\left( -\frac{\pi}{s^2}+\chi_3\right)
\right]\\
=\;& \frac{ (\sin\a)^s \, \eta}{\left(\pi+s\bar\kappa_2\right)^2 },
\end{split}\end{equation}
where we used the short notations~$\kappa_2=\kappa_2(\alpha,s)$,
$\bar\kappa_2:=\kappa_2\left(\frac\pi2,s\right)$ and
$$ \eta=\eta(\alpha,s):= s\chi_3\,(\pi-2\alpha+s\bar\kappa_2-s\kappa_2)
+\pi\kappa_2-2\alpha\bar\kappa_2.$$
We stress that an important simplification occurred in the last step
of~\eqref{AP2:prebis}.

Notice also that
\begin{equation}
\eta(\alpha,0)=\pi\kappa_2(\alpha,0)-2\alpha\kappa_2\left(\frac\pi2,0
\right).
\end{equation}
Furthermore, recalling~\eqref{RAP},
\begin{eqnarray*}
&& \frac{(\sin\a)^s\,\log
(\sin\a)\, I(1,\a,s) }{ I\left(1,\frac\pi2,s\right) } = \frac{(\sin\a)^s\,\log
(\sin\a)\, \left( \frac{2\a}{s}+\kappa_2 \right) }{
\frac{\pi}{s}+\bar\kappa_2  }=\frac{(\sin\a)^s\,\log
(\sin\a)\, \left( 2\a+s\kappa_2 \right) }{
\pi+s\bar\kappa_2  }
.\end{eqnarray*}
Using this, \eqref{AP2:pre} and~\eqref{AP2:prebis}, we conclude that
\begin{equation*}
\partial_s f(s,\a)=\frac{(\sin\a)^s\,\log
(\sin\a)\, \left( 2\a+s\kappa_2 \right) }{
\pi+s\bar\kappa_2  }+\frac{ (\sin\a)^s \,\eta}{\left(\pi+s\bar\kappa_2\right)^2 }.
\end{equation*}
Thus, recalling
also~\eqref{SINBO} and~\eqref{CALCO k2},
\begin{equation}\label{jd991}
\partial_s f(0,\theta(0))=\frac{2\theta(0)\,\log
\big(\sin\theta(0)\big) }{
\pi }-\frac{ 2\pi\Xi\big(\theta(0)\big)
-4\theta(0)\,\Xi\left( \frac\pi2\right)
 }{\pi^2 }.
\end{equation}
Now we observe that,
in view of~\eqref{IMV} and~\eqref{ka1},
$$ \partial_\a I(1,\alpha,s)=2\,
\int_0^{+\infty}
\frac{r\,dr}{
{\left( {r}^{2}+2\,r\,\cos \a+1\right) }^{\frac{2+s}{2}}}
=\frac2{s}+\tilde\kappa_1,$$
where~$\tilde\kappa_1=\tilde\kappa_1(\a,s)$ is a function
which is bounded uniformly
in~$\a\in\left[0,\frac\pi2\right]$ and~$s\in(0,1)$.
Consequently, recalling~\eqref{conS} and~\eqref{RAP},
\begin{eqnarray*}
\partial_\a f(s,\a)&=&
\frac{ s(\sin\a)^{s-1}\cos\a\, \left( \frac{2\a}{s}+\kappa_2\right)
+(\sin\a)^s\, \left( \frac2{s}+\tilde\kappa_1\right) }{ \frac{\pi}{s}+\bar\kappa_2}\\
&=&\frac{ s(\sin\a)^{s-1}\cos\a\, \left( {2\a}+s\kappa_2\right)
+(\sin\a)^s\, \left( 2+s\tilde\kappa_1\right) }{ \pi+s\bar\kappa_2}.
\end{eqnarray*}
Therefore, exploiting~\eqref{SINBO} once again,
$$
\partial_\a f(0,\theta(0))=\frac2{ \pi}.
$$
Making use of this, \eqref{DINI}
and~\eqref{jd991}, we find
\begin{equation*}
\partial_s\theta(0)=-\,\theta(0)\,\log\big(\sin\theta(0)\big)
+\frac{ \pi\Xi\big(\theta(0)\big)-2\theta(0)\,\Xi\left(\frac\pi2\right) }{\pi}.
\end{equation*}
Accordingly, by~\eqref{LINEAR},
\begin{equation*}
\partial_s\theta(0)=-\frac{\pi\,(1+\sigma)}{2}\,\log\left(
\cos\frac{\pi\sigma}{2}\right)
+\,\Xi\left(\frac{\pi(1+\sigma)}{2}\right)-(1+\sigma)\,
\Xi\left(\frac\pi2\right) .
\end{equation*}
{F}rom this and~\eqref{LINEAR},
the desired result in~\eqref{LASECONDA} plainly follows.
\medskip

Now we check the monotonicity of the function~$\theta(s,\sigma)$
with respect to~$\sigma$. For this, 	
we differentiate~\eqref{LASIG} (recall that now $\theta=\theta(s,\sigma)$)
and see that
\begin{equation} \label{yah91}
f_\alpha(s,\theta(s,\sigma))
\,\partial_\sigma\theta(s,\sigma)=1.\end{equation}
Also, from~\eqref{IMV}, we have that~$\partial_\a I(1,\a,s)>0$.
Accordingly, by~\eqref{conS}, we obtain that~$
\partial_\a f(s,\a)>0$.
This and~\eqref{yah91} give that~$\partial_\sigma\theta(s,\sigma)>0$,
as desired.
\end{proof}

\begin{appendix}

\section{Remarks on the shape of the minimizers}\label{09:SEC}

It is interesting to remark that minimizers of capillarity problems
with $\sigma=0$, $g=0$, $n=2$ and $\Omega=H=\{x_2>0\}$ are not half-balls,
differently to what happens in the classical case.

To check this statement, suppose, by contradiction, that $B=H B_\rho(0)$ is a critical
point (here and in the sequel, for typographical convenience,
we use the short
notation for intersection of sets~$AB:=A\cap B$). Let $x\in H\,\partial B$ and denote by $R$ the reflected half-ball with
respect to the tangent line to $\partial B$ at $x$. We observe that $R\subset H$
and $B^c\supset H^c$.
Then, by formula (1.23)
in \cite{MaggiValdinoci}, the Euler-Lagrange equation at any point $x\in H\,\partial B$
reads as
\begin{equation}\label{ID zero}
\begin{split}
0\,&= \int_{\R^2} \frac{1_{B^c}(y)-1_{B}(y)}{|x-y|^{2+s}}\,dy
-\int_{H^c} \frac{dy}{|x-y|^{2+s}}\\
&=\int_{B^c H} \frac{dy}{|x-y|^{2+s}}
-\int_{B} \frac{dy}{|x-y|^{2+s}} \\
&= \int_{R} \frac{dy}{|x-y|^{2+s}}+
\int_{B^c R^c H} \frac{dy}{|x-y|^{2+s}}
-\int_{B} \frac{dy}{|x-y|^{2+s}}\\
&= \int_{B^c R^c H} \frac{dy}{|x-y|^{2+s}} =: {\mathcal{F}}(x),
\end{split}
\end{equation}
where a cancellation due to the symmetry between $B$
and $R$ was used in the last step of this identity.

Now we evaluate \eqref{ID zero} at $p=(-\rho,0) $
(see Figure \ref{FIG zero 1}) and at $q=(0,\rho)$
(see Figure \ref{FIG zero 2}), we use some
geometric argument exploiting isometric regions of $B^c R^c H$
and we obtain the desired contradiction.

\begin{figure}
  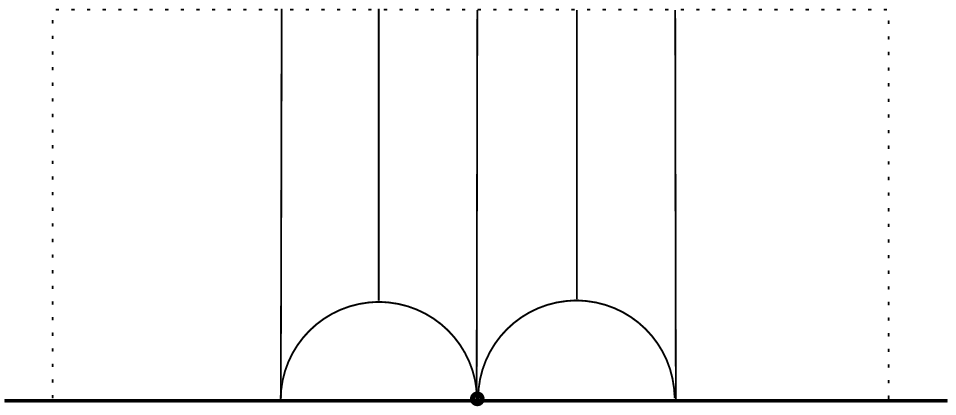\caption{{\small Formula \eqref{ID zero} evaluated at $p=(-\rho,0)$.}}
  \label{FIG zero 1}
\end{figure}

\begin{figure}
  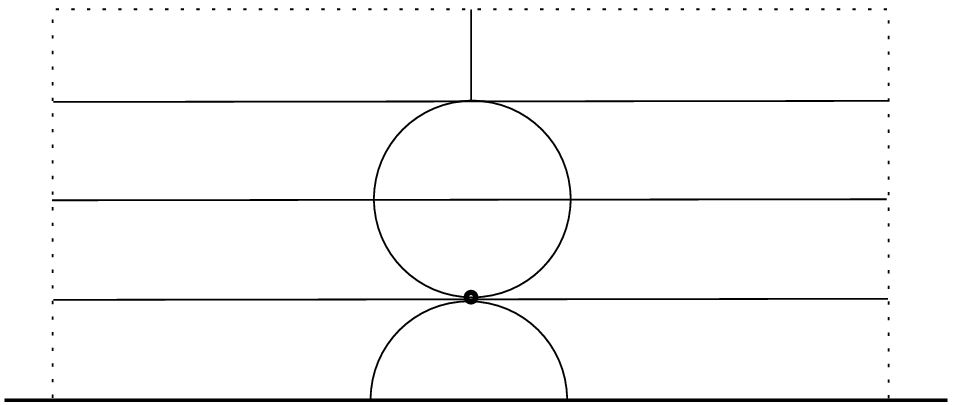\caption{{\small Formula \eqref{ID zero} evaluated at $q=(0,\rho)$.}}
  \label{FIG zero 2}
\end{figure}

%\begin{figure}
%    \centering
%    \includegraphics[scale=0.45]{uno.pdf}
%    \caption{Formula \eqref{ID zero} evaluated at $p=(-\rho,0) $.}
%    \label{FIG zero 1}
%\end{figure}

%\begin{figure}
%    \centering
%    \includegraphics[scale=0.45]{FIGb2.pdf}
%    \caption{Formula \eqref{ID zero} evaluated at $q=(0,\rho) $.}
%    \label{FIG zero 2}
%\end{figure}

To this aim,
we observe that, by \eqref{ID zero},
\begin{equation}\label{CON zero 1}
{\mathcal{F}}(p)=0={\mathcal{F}}(q).\end{equation}
On the other hand, we claim that
\begin{equation}\label{CON zero 2}
{\mathcal{F}}(p)<{\mathcal{F}}(q).\end{equation}
For this, we partition $B^c R^c H$ into the regions
$L$, $M$, $S$, $T$, $W$ and $X$ in
Figure \ref{FIG zero 1} and into the regions
$L$, $M$, $S$, $T$, $S'$, $T'$, $X$, $Y$ and $W$ in
Figure \ref{FIG zero 2}.

We observe that the contributions coming from $L$, $M$, $S$, $T$, $X$ and $W$
in Figure \ref{FIG zero 1} are, by isometry, exactly the same
as the ones coming from $L$, $M$, $S$, $T$, $X$ and $W$
in Figure \ref{FIG zero 2}.

Also, Figure \ref{FIG zero 2} possesses the additional contributions from $S'$, $T'$
and $Y$ that are not present in Figure \ref{FIG zero 1}. Therefore, the total
contributions in Figure \ref{FIG zero 2} are larger than the ones in
Figure \ref{FIG zero 1},
and
this proves \eqref{CON zero 2}.

Then, a contradiction arises by comparing \eqref{CON zero 1}
and \eqref{CON zero 2}, thus proving that half-balls are not critical
(and, in particular, not minimal).

\section{Proof of the asymptotics in \eqref{FG:A0}}\label{876}.

Up to scaling, we may assume that~$\Omega\subseteq B_1$.
Then, from formula~(3.9) in~\cite{MR3007726}, we have that
\begin{equation}\label{FG:AX}
\lim_{s\to0^+} s\,I_s(E,E^c\cap\Omega)\le
\lim_{s\to0^+} s\,\int_E\,dx\, \int_{B_1\setminus E}\frac{dy}{|x-y|^{n+s}}=0.
\end{equation}
Similarly, for any $R\ge1$,
\begin{equation}\label{01owriri}
\lim_{s\to0^+} s\,\int_E\,dx\, \int_{B_R\cap\Omega^c
}\frac{dy}{|x-y|^{n+s}}\le \lim_{s\to0^+} s\,\int_\Omega\,dx\,
\int_{B_R\cap\Omega^c
}\frac{dy}{|x-y|^{n+s}}=0.
\end{equation}
Also, from (2.2) in~\cite{MR3007726},
$$ \alpha(\Omega^c):=
\lim_{s\to0^+} s\,\int_{\Omega^c\cap B_1^c}\frac{dy}{|y|^{n+s}}=
\lim_{s\to0^+} s\,\int_{ B_1^c}\frac{dy}{|y|^{n+s}}=:\bar c(n).$$
This and formula (3.8) in~\cite{MR3007726} give that
\begin{eqnarray*}
0 &=& \lim_{R\to+\infty}\lim_{s\to0^+}
\left| \alpha(\Omega^c)\,|E|-
s\,\int_E\,dx\,\int_{\Omega^c\cap B_R^c}\frac{dy}{|x-y|^{n+s}}\right|\\
&=& \lim_{R\to+\infty}
\lim_{s\to0^+}
\left| \bar c(n)\,|E|-
s\,\int_E\,dx\,\int_{\Omega^c\cap B_R^c}\frac{dy}{|x-y|^{n+s}}\right|
.\end{eqnarray*}
{F}rom this and \eqref{01owriri} we obtain that
\begin{eqnarray*}
&& \lim_{s\to0^+}
\left| \bar c(n)\,|E|-
s\,\int_E\,dx\,\int_{\Omega^c}\frac{dy}{|x-y|^{n+s}}\right|
\\ &&\qquad\le
\lim_{R\to+\infty}
\lim_{s\to0^+}
\left| \bar c(n)\,|E|-
s\,\int_E\,dx\,\int_{\Omega^c\cap B_R^c}\frac{dy}{|x-y|^{n+s}}\right|
\\ &&\qquad\qquad+
s\,\int_E\,dx\,\int_{\Omega^c\cap B_R}\frac{dy}{|x-y|^{n+s}}=0,
\end{eqnarray*}
that is
$$ \lim_{s\to0^+}  s\,I_s(E,\Omega^c)=\bar c(n)\,|E|.$$
This and~\eqref{FG:AX} imply \eqref{FG:A0}.
\end{appendix}

\bibliography{references}
\bibliographystyle{is-alpha}
\def\cprime{$'$}

\end{document}